\documentclass[11pt,leqno]{amsart}
\usepackage{amsmath,amssymb,amsthm}
\usepackage{hyperref}

\usepackage{bbm}

\renewcommand{\geq}{\geqslant}
\renewcommand{\leq}{\leqslant}

\DeclareMathOperator{\dens}{dens}
\newcommand{\Lip}{{\mathrm{Lip}}_0}

\newtheorem{theorem}{Theorem}[section]
\newtheorem{lemma}[theorem]{Lemma}
\newtheorem{claim}[theorem]{Claim}

\theoremstyle{definition}
\newtheorem{definition}[theorem]{Definition}
\newtheorem{example}[theorem]{Example}

\theoremstyle{remark}
\newtheorem{remark}[theorem]{Remark}

\numberwithin{equation}{section}

\def\fnote#1{\footnote}

\def\ignora#1{}
\def\n3#1{\left\vert  \! \left\vert \! \left\vert \, #1 \, \right\vert \!
  \right\vert \! \right\vert }

\newcommand{\iten}{\ensuremath{\widehat{\otimes}_\varepsilon}}

\begin{document}

\title{ $L$-orthogonality in Daugavet centers and narrow operators }

\author{ Abraham Rueda Zoca }\thanks{ The author was supported by MICINN (Spain) Grant PGC2018-093794-B-I00 (MCIU, AEI, FEDER, UE), by Junta de Andaluc\'ia Grant A-FQM-484-UGR18 and by Junta de Andaluc\'ia Grant FQM-0185}
\address{Universidad de Murcia, Departamento de Matem\'aticas, Campus de Espinardo 30100 Murcia, Spain} \email{ abrahamrueda@ugr.es}
\urladdr{\url{https://arzenglish.wordpress.com}}



\maketitle

\begin{abstract}
We study the presence of $L$-orthogonal elements in connection with Daugavet centers and narrow operators. We prove that, if $\dens(Y)\leq \omega_1$ and $G:X\longrightarrow Y$ is a Daugavet center, then $G(W)$ contains some $L$-orthogonal for every non-empty $w^*$-open subset of $B_{X^{**}}$. In the context of narrow operators, we show that if $X$ is separable and $T:X\longrightarrow Y$ is a narrow operator, then given $y\in B_X$ and any non-empty $w^*$-open subset $W$ of $B_{X^{**}}$ then $W$ contains some $L$-orthogonal $u$ so that $T^{**}(u)=T(y)$. In the particular case that $T^*(Y^*)$ is separable, we extend the previous result to $\dens(X)=\omega_1$. Finally, we prove that none of the previous results holds in larger density characters (in particular, a counterexample is shown for $\omega_2$ under continuum hypothesis).
\end{abstract}

\section{Introduction}

A Banach space $X$ is said to have the Daugavet property if every rank-one operator $T:X\longrightarrow X$ satisfies the equality
\begin{equation}\label{ecuadauga}
\Vert T+I\Vert=1+\Vert T\Vert,
\end{equation}
where $I$ denotes the identity operator. The previous equality is known as \textit{Daugavet equation} because I. Daugavet proved in \cite{dau} that every compact operator on $\mathcal C([0,1])$ satisfies (\ref{ecuadauga}). Since then, a lot of examples of Banach spaces enjoying the Daugavet property have appeared such as $\mathcal C(K)$ for a compact Hausdorff and perfect topological space $K$, $L_1(\mu)$ and $L_\infty(\mu)$ for a non-atomic measure $\mu$ or the space of Lipschitz functions $\Lip(M)$ over a metrically convex space $M$ (see \cite{ikw,kssw,shv,werner} and the references therein for details). Moreover, in \cite{kssw} (respectively \cite{shv}) a characterisation of the Daugavet property in terms of the geometry of the slices (respectively non-empty weakly open subsets) of $B_X$ appeared. Namely, a Banach space $X$ has the Daugavet property if, and only if, given any $x\in S_X$, any non-empty weakly open subset $W$ of $B_X$ and any $\varepsilon>0$ there exists $y\in W$ such that $\Vert x+y\Vert>2-\varepsilon$. Despite the strong geometric spirit of this characterisation, it has been useful to extend the Daugavet equation from rank one operators to wider classes of operators such as operators which do not fix any copy of $\ell_1$ \cite[Theorem 3]{shv}. 

Very recently, the previous geometric characterisation of the Daugavet property was put further in the following sense: let $X$ be a Banach space with the Daugavet property. Is it true that, given any $w^*$-open subset $W$ of $B_{X^{**}}$, there exists $v\in W$ so that
\begin{equation}\label{eq:Lortogonal}\Vert x+v\Vert=1+\Vert x\Vert
\end{equation} 
holds for every $x\in X$?

The motivation for studying this question came from the paper \cite[Lemma 9.1]{gk}, where it is proved that, if $X$ is a separable Banach space, the existence of $u\in S_{X^{**}}$ satisfying \eqref{eq:Lortogonal} is equivalent to the fact that \textit{the norm of $X$ is octahedral} (see Remark \ref{remark:octanoempint} for a formal definition). Since it is well known that the Daugavet property implies octahedrality of the norm \cite[Lemma 2.8]{kssw}, the author proved in \cite[Theorem 3.2]{rueda} that, if $X$ is a separable Banach space with the Daugavet property, then the set of those $v\in S_{X^{**}}$ satisfying \eqref{eq:Lortogonal} (and called \textit{L-orthogonal elements} in \cite{loru}) is $w^*$-dense in $B_{X^{**}}$. Apart from being a natural extension of the Daugavet property, this abundance of $L$-orthogonal elements showed to be useful to study $L$-embedded Banach spaces with the Daugavet property \cite[Theorem 3.4]{rueda} and to study the Daugavet property in projective tensor products of an $L$-embedded Banach space \cite[Theorem 3.7]{rueda}. 

Motivated by the previous result and by the question whether \cite[Lemma 9.1]{gk} holds in non-separable cases, in \cite{loru} it was proved that the characterisation of Daugavet property in terms of abundance of elements satisfying \eqref{eq:Lortogonal} is characteristic of Banach spaces with small density character. Indeed, in \cite[Theorem 3.6]{loru} it is proved that if $X$ is a Banach space with the Daugavet property and $\dens(X)=\omega_1$ then there are a $w^*$-dense subset of $L$-orthogonal elements in $B_{X^{**}}$. However, there are Banach spaces $X$ with $\dens(X)=card(\mathcal P(\mathbb R))$ having the Daugavet property but for which there is no element $v\in S_{X^{**}}$ satisfying \eqref{eq:Lortogonal} (in particular, this proves that the above mentioned \cite[Theorem 3.6]{loru} is sharp under the continuum hypothesis) \cite[Example 3.8]{loru}.

From here two consequences are derived. On the one hand, as we have pointed out, the presence of $L$-orthogonality in connection with the Daugavet property is a phenomenon that does not happen in large density characters. However, for $\dens(X)\leq \omega_1$, there is a strong interplay between the Daugavet property and the abundance of elements $L$-orthogonal to $X$ which goes further from octahedrality because, to the best of our knownledge, the question whether or not the result \cite[Lemma 9.1]{gk} holds for Banach spaces of density character equal to $\omega_1$ remains open.

In view of this fact, in this paper we aim to study how deep the $L$-orthogonality is connected with the Daugavet property and the Daugavet equation. To do so, we will study its presence in connection with two important concepts coming from the Daugavet equation: the concept of Daugavet center and the concept of narrow operator. 

After introducing necessary notation and preliminary results in Section \ref{sect:notation}, in Section \ref{sect:daucen} we will make a study of the Daugavet centers. The main result is Theorem \ref{theo:daucenterome1}, where we prove that if $G:X\longrightarrow Y$ is a Daugavet center with $G(X)$ separable and $\dens(Y)\leq \omega_1$ then, given any non-empty $w^*$-open subset $W$ of $B_{X^{**}}$, there exists $u\in W$ so that $\Vert G^{**}(u)+y\Vert=1+\Vert y\Vert$ holds for every $y\in Y$. Based on \cite[Example 3.8]{loru}, we show in Example \ref{examp:countercenter} that the previous result does not hold if $\dens(Y)$ is larger.

In Section \ref{sect:narrop} we study the presence of $L$-orthogonality for narrow operators. Here we obtain two different results. In Theorem \ref{theo:sepasrLornarrow} we prove that if $X$ is a separable Banach space and $T:X\longrightarrow Y$ is a narrow operator, then given any $y\in B_X$ and any non-empty $w^*$-open subset $W$ containing $y$ there exists an element $u\in W$ with $T^{**}(u)=T^{**}(y)$ and such that $\Vert x+u\Vert=1+\Vert x\Vert$ holds for every $x\in X$. When $T^*(Y^*)$ is separable, we prove in Theorem \ref{theo:nosepLordualsep} the same result when $\dens(X)=\omega_1$. Again, we show in Example \ref{exam:counnarrow} that the previous result is no longer true when $\dens(X)$ is larger.

\section{Notation and preliminary results}\label{sect:notation}

We will consider only real Banach spaces. Given a Banach space $X$, we will denote the unit ball and the unit sphere of $X$ by $B_X$ and $S_X$ respectively. Moreover, given $x\in X$ and $r>0$, we will denote $B(x,r)=x+rB_X=\{y\in X:\Vert x-y\Vert\leq r\}$. We will also denote by $X^*$ the topological dual of $X$. Given a bounded subset $C$ of $X$, we will mean by a \textit{slice of $C$} a set of the following form
$$S(C,x^*,\alpha):=\{x\in C:x^*(x)>\sup x^*(C)-\alpha\}$$
where $x^*\in X^*$ and $\alpha>0$. If $X$ is a dual Banach space, the previous set will be called a \textit{$w^*$-slice} if $x^*$ belongs to the predual of $X$. Note that finite intersections of slices of $C$ (respectively of $w^*$-slices of $C$) form a basis for the inherited weak (respectively weak-star) topology of $C$. Throughout the text $\omega_1$ (respectively $\omega_2$) will denote the first uncountable ordinal (respectively the first ordinal whose cardinal is strictly bigger than the cardinality of $\omega_1$).

Let $Z$ be a subspace of a Banach space $X$.
We say that $Z$ is an \emph{almost isometric ideal} (ai-ideal) in $X$ if
$X$ is locally complemented in $Z$ by almost isometries.
This means that for each $\varepsilon>0$ and for each
finite-dimensional subspace $E\subseteq X$ there exists a linear
operator $T:E\to Z$ satisfying
\begin{enumerate}
\item\label{item:ai-1}
  $T(e)=e$ for each $e\in E\cap Z$, and
\item\label{item:ai-2}
  $(1-\varepsilon) \Vert e \Vert \leq \Vert T(e)\Vert\leq
  (1+\varepsilon) \Vert e \Vert$
  for each $e\in E$,
\end{enumerate}
i.e. $T$ is a $(1+\varepsilon)$ isometry fixing the elements of $E$.
If the $T$ satisfies only (\ref{item:ai-1}) and the right-hand side of
(\ref{item:ai-2}) we get the well-known
concept of $Z$ being an \emph{ideal} in $X$ \cite{gks}.

Note that the Principle of Local Reflexivity means that $X$ is an ai-ideal in $X^{**}$
for every Banach space $X$.

Throughout the text we will make use of the following two results, which we include here for the sake of completeness and for easy reference.

\begin{theorem}\label{theo:hbaioper}\cite[Theorem 1.4]{aln2}
Let $X$ be a Banach space and let $Z$ be an almost isometric ideal in $X$. Then there is a linear isometry $\varphi: Z^*\longrightarrow X^*$ such that
$$\varphi(z^*)(z)=z^*(z)$$
holds for every $z\in Z$ and $z^*\in Z^*$ and satisfying that, for every $\varepsilon>0$, every finite-dimensional subspace $E$ of $X$ and every finite-dimensional subspace $F$ of $Z^*$, we can find an operator $T:E\longrightarrow Z$ satisfying
\begin{enumerate}
\item $T(e)=e$ for every $e\in E\cap Z$, 
\item $(1-\varepsilon)\Vert e\Vert\leq \Vert T(e)\Vert\leq (1+\varepsilon)\Vert e\Vert$ holds for every $e\in E$, and;
\item $f(T(e))=\varphi(f)(e)$ holds for every $e\in E$ and every $f\in F$.
\end{enumerate}
\end{theorem}

Following the notation of \cite{abrahamsen}, to such an operator $\varphi$ we will refer as \textit{an almost-isometric Hahn-Banach extension operator}. Notice that if $\varphi:Z^*\longrightarrow X^*$ is an almost isometric Hahn-Banach extension operator, then $\varphi^*:X^{**}\longrightarrow Z^{**}$ is a norm-one projection (see e.g. \cite[Theorem 3.5]{kaltonlocomp}).

Another central result in our main theorems will be the following, coming from \cite[Theorem 1.5]{abrahamsen}

\begin{theorem}\label{theo:exteaiideales}
Let $X$ be a Banach space, let $Y$ be a separable subspace of $X$ and let $W\subseteq X^*$ be a separable subspace. Then there exists a separable almost isometric ideal $Z$ in $X$ containing $Y$ and an almost isometric Hahn-Banach extension operator $\varphi:Z^*\longrightarrow X^*$ such that $\varphi(Z^*)\supset W$.
\end{theorem}

According to \cite{gk}, given a Banach space $X$, the \textit{ball topology}, denoted by $b_X$, is defined as the coarsest topology on $X$ so that every closed ball is closed in $b_X$. As a consequence, a basis for the topology $b_X$ is formed by the sets of the following form
$$X\setminus\bigcup_{i=1}^n B(x_i,r_i),$$
where $x_1,\ldots, x_n$ are elements of $X$ and $r_1,\ldots, r_n$ are positive numbers.

Let us extract the following result from the proof of \cite[Lemma 9.1]{gk}, which will be used several times throughout the text.

\begin{lemma}\label{lema:gklemabxabier}
Let $X$ be a separable Banach space and let $\{O_n\}$ be a separable basis for the $b_X$-topology of $B_X$. If a sequence $\{x_n\}\subseteq S_X$ satisfies that
$$x_n\in\bigcap\limits_{k=1}^n O_k$$
holds for every $n\in\mathbb N$, then $\{x_n\}$ has a subsequence (say $\{x_{\sigma(n)}\}$) satisfying that, if $u$ is any $w^*$-cluster point of $\{x_{\sigma(n)}\}$, then
$$\Vert x+u\Vert=1+\Vert x\Vert$$
holds for every $n\in\mathbb N$.
\end{lemma}

\section{Daugavet centers}\label{sect:daucen}

Let us start with one central definition of the section.

\begin{definition}\label{defi:daugacenter}
Let $X$ and $Y$ be two Banach spaces and $G:X\longrightarrow Y$ be a bounded operator. We say that $G$ is a \textit{Daugavet center} if 
$$\Vert G+T\Vert=\Vert G\Vert+\Vert T\Vert$$
holds for every rank-one operator $T:X\longrightarrow Y$.
\end{definition}

Note that a Banach space $X$ has the Daugavet property if, and only if, the identity operator $I:X\longrightarrow X$ is a Daugavet center. Daugavet centers were introduced in \cite{bk} to generalise one known result about existence of equivalent renormings with the Daugavet property. Namely, it is proved that if $G:X\longrightarrow Y$ is a Daugavet center, $Y$ is a subspace of $E$ and $J:Y\longrightarrow E$ is the natural embedding operator, then $E$ admits and equivalent renorming so that $J\circ G:X\longrightarrow E$ is also a Daugavet center \cite[Theorem 1.3]{bk}. See \cite{bk,iva,santos2020} and references therein for background and examples of Daugavet centers.

Let us write a characterisation of Daugavet centers, coming from \cite{bk}.

\begin{lemma}\label{lemma:centrefindimsub}
Let $X$ and $Y$ be two Banach spaces and $G:X\longrightarrow Y$ be a Daugavet center. Then, given a finite-dimensional subspace $E$ of $Y$, a non-empty relatively weakly open subset $U$ of $B_X$ and $\varepsilon>0$, we can find another non-empty relatively weakly open subset $V\subseteq U$ such that
$$\Vert e+\lambda G(v)\Vert>(1-\varepsilon)(\Vert e\Vert+\vert \lambda \vert)$$
holds for every $ v\in V$, every $e\in E$ and every $\lambda\in\mathbb R$.
\end{lemma}

\begin{remark}\label{remark:octanoempint}
Recall that the norm of a Banach space $X$ is \textit{octahedral} if, given a finite-dimensional subspace $E$ of $X$ and $\varepsilon>0$, there exists $x\in S_X$ such that
$$\Vert e+\lambda x\Vert>(1-\varepsilon)(\Vert e\Vert+\vert\lambda\vert)$$
holds for every $e\in E$ and every $\lambda\in\mathbb R$. See \cite{god,gk} for background. Note that Lemma \ref{lemma:centrefindimsub} implies that if $G:X\longrightarrow Y$ is a Daugavet center then the norm of $Y$ is octahedral. Consequently, if $O_1, O_2$ are non-empty $b_Y$-open subsets of $B_Y$, then $O_1\cap O_2\neq\emptyset$ \cite[Lemma 9.1]{gk}. 
\end{remark}

With the previous lemma in mind let us prove the following preliminary lemma.

\begin{lemma}\label{lemma:daugacenterbxopen}
Let $X$ and $Y$ be Banach spaces, let $G:X\longrightarrow Y$ be a Daugavet center and $O$ be a non-empty $b_Y$-open subset of $B_Y$. Then, given any non-empty relatively weakly open subset $W$ of $B_X$ we get that
$$G(W)\cap O\neq\emptyset.$$
\end{lemma}

\begin{proof}
Up to take a smaller $b_Y$ open subset, we can assume that $O:=\bigcap\limits_{i=1}^k B_X\setminus B(x_i,r_i)$. Notice that, since $O$ is non-empty then $r_i<1+\Vert x_i\Vert$ holds for every $1\leq i\leq k$. Pick $0<\varepsilon<\min\limits_{1\leq i\leq k} 1+\Vert x_i\Vert-r_i$. Pick also a non-empty relatively weakly open subset $W$ of $B_X$. By Lemma \ref{lemma:centrefindimsub} we can find $x\in W$ with 
$$\Vert x_i+G(x)\Vert>1+\Vert x_i\Vert-\varepsilon$$
holds for every $i\in\{1,\ldots, k\}$. Given $i\in\{1,\ldots, k\}$ we get that
$$\Vert x_i+G(x)\Vert>1-\Vert x_i\Vert-\varepsilon>1+\Vert x_i\Vert-(1+\Vert x_i\Vert-r_i)=r_i,$$
so $G(x)\in B_X\setminus B(x_i,r_i)$. Since $i$ was arbitrary $G(x)\in G(W)\cap O$, and the proof is complete.
\end{proof}

Now we are ready to prove one of the main results in this part.

\begin{theorem}\label{theo:daugacensep}
Let $X$ and $Y$ be two Banach spaces and $G:X\longrightarrow Y$ be a Daugavet center. Assume that $Y$ is separable. Then, given $u\in B_{X^{**}}$ and $\{g_n:n\in\mathbb N\}\subseteq S_{X^*}$, we can find $v\in S_{X^{**}}$ such that
\begin{enumerate}
\item $v(g_n)=u(g_n)$ holds for every $n\in\mathbb N$ and,
\item\label{ortosepdaucen} $\Vert G^{**}(v)+y\Vert=1+\Vert y\Vert$ holds for every $y\in Y$.
\end{enumerate}
In particular, given any non-empty relatively $w^*$-open subset $W$ of $B_{X^{**}}$ there exists $v\in W$ satisfying \eqref{ortosepdaucen}.
\end{theorem}

\begin{proof}
Pick $\{O_n\}_{n\in\mathbb N}$ to be a basis of the $b_Y$ topology of $B_Y$. Since $G$ is a Daugavet center, by Lemma \ref{lemma:daugacenterbxopen} we can find $x_n\in\{z\in B_X: \vert g_i(z)-G^{**}(u)\vert<\frac{1}{n}, 1\leq i\leq n\}$, which is a non-empty relatively weakly open subset of $B_X$, such that $G(x_n)\in \bigcap\limits_{i=1}^n O_i$ (note that an easy inductive argument together with Remark \ref{remark:octanoempint} implies that $\bigcap\limits_{i=1}^n O_i$ is non-empty for every $n\in\mathbb N$). By Lemma \ref{lema:gklemabxabier} there exists a subsequence of $G(x_n)$, which we will denote in the same way, so that any $w^*$-limit point is $u$ satisfies that $\Vert y+u\Vert=1+\Vert y\Vert$ holds for every $y\in Y$. Pick $v\in \{x_n\}'$, where the last accumulation is in the $w^*$-topology of $B_{X^{**}}$. It is obvious that $v(g_n)=u(g_n)$ holds for every $n\in\mathbb N$. Moreover, since $G^{**}$ is $w^*-w^*$-continuous we get that $G^{**}(v)$ is a limit point of $\{G(x_n)\}$. From here $\Vert G^{**}(v)+y\Vert=1+\Vert y\Vert$ holds for every $y\in Y$, and the proof is complete.
\end{proof}

Let us now generalise the previous theorem to the case when $\dens(Y)=\omega_1$.

\begin{theorem}\label{theo:daucenterome1}
Let $X$ and $Y$ be two Banach spaces and $G:X\longrightarrow Y$ be a Daugavet center such that $G(X)$ is separable. If $\dens(Y)=\omega_1$ then, for every $u\in B_{X^{**}}$ and every $\{g_n: n\in\mathbb N\}\subseteq S_{X^*}$ we can find $v\in B_{X^{**}}$ such that
\begin{enumerate}
\item $v(g_n)=u(g_n)$ holds for every $n\in\mathbb N$ and,
\item\label{ortome1daucen} $\Vert G^{**}(v)+y\Vert=1+\Vert y\Vert$ holds for every $y\in Y$.
\end{enumerate}
In particular, given any non-empty relatively $w^*$-open subset $W$ of $B_{X^{**}}$ there exists $v\in W$ satisfying \eqref{ortome1daucen}.
\end{theorem}

\begin{proof}
The proof will follow the lines of \cite[Theorem 3.3]{loru}. Pick $\{y_\alpha\}_{\alpha<\omega_1}$ be a dense subset of $S_Y$. Let us construct, by transfinite induction, a family $\{(Z_\alpha,\varphi_\alpha, \{f_{\alpha,\beta}: \beta<\alpha\}, v_\alpha): \alpha\in\omega_1\}$ satisfying that, for every $\alpha<\omega_1$, then:
\begin{enumerate}
\item $Z_\alpha$ is a separable almost isometric ideal in $Y$ containing $\bigcup\limits_{\beta<\alpha}Z_\beta
\cup\{y_\alpha\}$ and $G(X)\subseteq Z_\alpha$.
\item $\varphi_\alpha:Z_\alpha^*\longrightarrow Y^*$ is an almost-isometric Hahn-Banach extension operator such that $\varphi_\alpha(Z_\alpha^*)\supset \{f_{\beta,\gamma}: \gamma<\beta<\alpha\}$.
\item The equality
$$\Vert G^{**}(v_\alpha)+y\Vert=1+\Vert y\Vert$$
holds for every $y\in Z_\alpha$, and $\{f_{\alpha,\beta}: \beta<\alpha\}\subseteq S_{Y^*}$ is norming for $Z_\alpha\oplus \mathbb R v_\alpha$.
\item $v_\alpha(G^*(f_{\beta,\gamma}))=v_\beta(G^*(f_{\beta,\gamma}))$ holds for every $\gamma<\beta<\alpha$ and
$$v_\alpha(g_n)=u(g_n)$$
holds for every $n\in\mathbb N$.
\end{enumerate}

The case $\alpha=\omega_0$ follows finding, in virtue of Theorem \ref{theo:exteaiideales}, a separable almost isometric ideal $Z_\alpha$ containing $G(X)\cup\{y_\beta: \beta\leq \omega_0\}$. Notice that, if we consider $G_\alpha:X\longrightarrow Z_\alpha$ to be the restriction to the codomain, then $G_\alpha$ is clearly a Daugavet center. An application of Theorem \ref{theo:daugacensep} yields an element $v_\alpha\in S_{X^{**}}$ so that $v_\alpha(g_n)=u(g_n)$ holds for every $n\in\mathbb N$ and such that
$$\Vert G_\alpha^{**}(v_\alpha)+y\Vert=1+\Vert y\Vert$$
holds for every $y\in Z_\alpha$. Pick $\{g_{\alpha,\beta}: \beta<\alpha\}\subseteq S_{Z_{\alpha}^*}$ to be a norming subset for $Z_\alpha\oplus \mathbb R G_\alpha^{**}(v_\alpha)$ and define $f_{\alpha,\beta}:=\varphi_\alpha^*(g_{\alpha,\beta})\in S_{Y^*}$. The proof of the case $\alpha=\omega_0$ will be complete when we prove the following:
\begin{claim}
For every $z^*\in Z_\alpha^*$ the following equality
$$G_\alpha^{**}(v_\alpha)(z^*)=G^{**}(v_\alpha)(\varphi_\alpha(z^*))$$
holds.
\end{claim}

\begin{proof}
Pick $z^*\in Z_\alpha^*$ and pick a net $(x_s)\subseteq B_{X^{**}}$ such that $x_s\rightarrow v_\alpha$ in the weak-star topology of $B_X$. Then $G(x_s)\rightarrow G^{**}(v_\alpha)$ in the weak-star topology of $B_{Y^{**}}$ because $G^{**}$ is $w^*-w^*$ continuous. In particular, $G(x_s)(\varphi_\alpha(z^*))\rightarrow G^{**}(v_\alpha)(\varphi_\alpha(z^*))$. A similar argument shows that $G_\alpha(x_s)(z^*)\rightarrow G_\alpha^{**}(v_\alpha)(z^*)$. Now, given $s$, taking into account that $G(x_s)\in G(X)\subseteq Z_\alpha$ (and so $G_\alpha(x_s)=G(x_s)$) and the fact that $\varphi_\alpha$ is a Hahn-Banach extension operator we deduce that
$$G_\alpha(x_s)(z^*)=\varphi_\alpha(z^*)(G(x_s)).$$
Using the uniqueness of the limit the claim is proved.
\end{proof}
Summarising, we have constructed $(Z_{\omega_0},\varphi_{\omega_0},\{f_{\omega_0,\beta}:\beta<\omega_0\},v_{\omega_0})$.

Now assume by induction hypothesis that $\{(Z_\beta,\varphi_\beta, \{f_{\beta,\gamma}: \gamma<\beta\}, v_\beta)\}_{\beta<\alpha}$ has been constructed, and let us construct $(Z_\alpha, \varphi_\alpha, \{f_{\alpha,\gamma}: \gamma<\alpha\},v_\alpha)$. To this end, pick a cluster point $v\in B_{X^{**}}$ of the net $\{v_\beta\}_{\beta<\alpha}$ (where we consider on $[0,\alpha[$ the classical order). Notice that, by induction hypothesis and the cluster point condition, we get that
$$v(G^*(f_{\beta,\gamma}))=v_\beta(G^*(f_{\beta,\gamma}))$$
holds for every $\gamma<\beta<\alpha$ and
$$v(g_n)=u(g_n)$$
holds for every $n\in\mathbb N$.

 Find, again by Theorem \ref{theo:exteaiideales}, a separable almost isometric ideal $Z_\alpha$ in $Y$ containing the separable set $\bigcup\limits_{\beta<\alpha} Z_\beta\cup\{y_\alpha\}$  and take an almost-isometric Hanh-Banach extension operator $\varphi_\alpha:Z_\alpha^*\longrightarrow Y^*$ such that $\varphi_\alpha(Z_\alpha^*)$ contains the countable set $\{f_{\beta,\gamma}: \gamma<\beta<\alpha\}$. By the same argument as in the case $\omega_0$, by an application of Theorem \ref{theo:daugacensep} we find $v_\alpha\in B_{X^{**}}$ such that 
\begin{enumerate}
\item $v_\alpha(G^*(f_{\beta,\gamma}))=v(G^*(f_{\beta,\gamma}))$ for $\gamma<\beta$ and $v_\alpha(g_n)=v(g_n)$ holds for every $n\in\mathbb N$.
\item $\Vert G^{**}(v_\alpha)+y\Vert=1+\Vert y\Vert$ holds for every $y\in Y$.
\end{enumerate}
Pick $\{f_{\alpha,\beta}: \beta<\alpha\}\subseteq S_{Y^*}$ to be a norming set for $Z_\alpha\oplus \mathbb R G^{**}(v_\alpha)$.
All the above proves that $(Z_\alpha, \varphi_\alpha, \{f_{\alpha,\gamma}: \gamma<\alpha\},v_\alpha)$ satisfies our requirements and completes the inductive argument.

Now we have proved the existence of the chain $(Z_\alpha, \varphi_\alpha, \{f_{\alpha,\gamma}: \gamma<\alpha\},v_\alpha)_{\alpha<\omega_1}$. Take $v\in B_{X^{**}}$ to be a limit point of the net $\{v_\alpha\}_{\alpha<\omega_1}$ (where we again consider on $[0,\omega_1[$ the classical order). It is inmediate that $v(g_n)=u(g_n)$ holds for every $n\in\mathbb N$. Let us prove, to finish, that 
$$\Vert G^{**}(v)+y\Vert=1+\Vert y\Vert$$
holds for every $y\in Y$. By a homogeneity argument, we can assume with no loss of generality that $y\in S_Y$. To this end, pick $\varepsilon>0$ and find, by the denseness of $\{y_\alpha:\alpha<\omega_1\}$, an element $y_\alpha$ so that $\Vert y-y_\alpha\Vert<\frac{\varepsilon}{3}$. Since $\Vert G^{**}(v_\alpha)+y_\alpha\Vert=2$ and $\{f_{\alpha,\beta}: \beta<\alpha\}$ is norming for $Z_\alpha\oplus G^{**}(v_\alpha)$ we can find $\beta<\alpha$ so that
$$G^{**}(v_\alpha)(f_{\alpha,\beta})+f_{\alpha,\beta}(y_\alpha)>2-\frac{\varepsilon}{3}.$$
Since $G^{**}(v_\delta)(f_{\alpha,\beta})=G^{**}(v_\alpha)(f_{\alpha,\beta})$ holds for every $\delta>\alpha$ and it is clear because of the $w^*-w^*$ continuity of $G^{**}$ that $G^{**}(v)$ is a cluster point of the net $(G^{**}(v_\delta))_{\delta<\omega_1}$, we conclude that $G^{**}(v)(f_{\alpha,\beta})=G^{**}(v_\alpha)(f_{\alpha,\beta})$ and so
$$\Vert y_\alpha+G^{**}(v)\Vert\geq 2-\frac{\varepsilon}{3},$$
thus $\Vert G^{**}(v)+y\Vert>2-\varepsilon$. Since $\varepsilon>0$ was arbitrary we conclude that $\Vert G^{**}(v)+y\Vert=2$, and the proof is complete.
\end{proof}

\begin{remark}\label{rema:caradaucenter}
The converse of Theorem \ref{theo:daucenterome1} is true in complete generality (and so it is actually a complete characterisation of when an operator with separable image is a Daugavet center via $L$-orthogonality). To be more precise, assume that $G:X\longrightarrow Y$ is an operator such that, for every non-empty $w^*$-open subset of $B_{X^{**}}$ there exists $u\in W$ with $\Vert G^{**}(u)+y\Vert=1+\Vert y\Vert$ holds for every $y\in Y$. Then $G$ is a Daugavet center.

In order to see it, fix $y_0\in S_Y$, $\varepsilon>0$ and a slice $S=S(B_X,f,\alpha)$. Since $S(B_{X^{**}},f,\alpha)$ is non-empty and $w^*$-open, by assumption there exists $u\in S(B_{X^{**}},f,\alpha)$ such that $\Vert G^{**}(u)+y\Vert=1+\Vert y\Vert$ holds for ever $y\in Y$. Find a net $\{x_s\}$ in $B_X$ weak-star convergent to $u$. The $w^*-w^*$ continuity of $G^{**}$ implies that $G(x_s)=G^{**}(x_s)$ coverges to $G^{**}(u)$ in the $w^*$ topology. The $w^*$-lower semicontinuity of the norm of $Y^{**}$ implies that
$$2=\Vert G^{**}(u)+y_0\Vert\leq \liminf_s
\Vert G(x_s)+y\Vert,$$
so we can find $s$ large enough so that $f(x_s)>1-\alpha$ (i.e. $x_s\in S$) and $\Vert G(x_s)+y_0\Vert>2-\varepsilon$. According to \cite[Theorem 2.1 (iii)]{bk}, $G$ is a Daugavet center.
\end{remark}

We do not know whether Theorem \ref{theo:daucenterome1} remains true if we remove the assumption that $T(X)$ is separable. In order to make use of Theorem \ref{theo:daucenterome1} together with an inductive argument we would need any result which guarantee that the codomain restriction of a Daugavet center is a Daugavet center.

However, we know that $\dens(Y)=\omega_1$ can not be removed in general, as the following example shows.

\begin{example}\label{examp:countercenter}
Take $X=Y=\ell_2(\mathcal P(\mathbb R))\iten C([0,1])$. Then $X$ has the Daugavet property \cite[P. 81]{werner}, but there exists no element $v\in S_{X^{**}}$ such that
$$\Vert x+v\Vert=1+\Vert x\Vert$$
holds for every $x\in X$ \cite[Example 3.8]{loru}. Hence, the identity operator $I:X\longrightarrow X$ is a Daugavet center which does not satisfies the thesis of Theorem \ref{theo:daucenterome1}. 

Note that, under continuum hypothesis, $\dens(\ell_2(\mathcal P(\mathbb R)))=card(\mathcal P(\mathbb R))=\omega_2$, and so the result is sharp.\end{example}

\begin{remark}
Very recently, E. R. Santos considered in \cite{santos2020} the concept of \textit{polynomial Daugavet center}. Given two Banach spaces $X$ and $Y$, it is said that a norm-one polynomial $Q:X\longrightarrow Y$ is a \textit{polynomial Daugavet center} if the Daugavet equation
$$\Vert Q+P\Vert=1+\Vert P\Vert$$
holds for every continuous rank-one polynomial $P\in \mathcal P(X,Y)$. After the proof of \cite[Proposition 2.6]{santos2020}, the author posed the following question: if $Q$ is a polynomial Daugavet center and $S=S(P,\alpha)$ is a polynomial slice of $B_{X^{**}}$, is there any $u\in S\cap S_{X^{**}}$ such that
\begin{equation}\label{eq:condisantos}
 \Vert \hat Q(u)+sign(P(u))y\Vert=1+\Vert y\Vert\end{equation}
holds for every $y\in Y$?

Incidentaly, Remark \ref{examp:countercenter} also provides a negative answer to Santos' question. Indeed, given $X=\ell_2(\mathcal P(\mathbb R))\iten C([0,1])$, the identity operator $I:X\longrightarrow X$ is actually a polynomial Daugavet center because $X$ has the \textit{polynomial Daugavet property} \cite[Corollary 2.5]{cgmm} which, in the languaje of polynomial Daugavet centers, means nothing but that $I$ is a polynomial Daugavet center. However, there is no $u$ satisfying equation \ref{eq:condisantos} because there exists no $u\in S_{X^{**}}$ such that
$$\Vert x+u\Vert=1+\Vert x\Vert$$
holds for every $x\in X$.
\end{remark}

\section{Narrow operators}\label{sect:narrop}

Different notions of narrow operators have appeared in the literature for concrete classes of Banach spaces (see \cite[Section 4]{werner}). Here we will consider the general notion considered in \cite[Section 4]{werner}.

\begin{definition}\label{defi:narrowop}
Let $X$ and $Y$ be two Banach spaces. We say that a bounded operator $T:X\longrightarrow Y$ is \textit{narrow} if, given $x,y\in S_X$, $\varepsilon>0$ and a slice $S$ containing $y$, there exists $z\in S$ such that $\Vert x+z\Vert>2-\varepsilon$ and $\Vert T(y-z)\Vert<\varepsilon$.
\end{definition}

It is clear, by the celebrated characterisation of the Daugavet property from \cite[Lemma 2.1]{kssw}, that a Banach space $X$ has the Daugavet property if, and only if, there exists a narrow operator from $X$ to any Banach space. One reason why narrow operators are interesting that if $T:X\longrightarrow X$ is a narrow operator then $T$ satisfies the Daugavet equation (see e.g. \cite[Lemma 4.3]{werner}). Another interest of narrow operators is that they are useful in order to pass the Daugavet property from a space to a subspace (see \cite[Section 5]{werner} for background on \textit{rich subspaces}). See \cite{akmms,kkw,kkw2,ksew,ksw,werner} and references therein for background.

Let us announce the following result, coming from \cite[Proposition 4.12]{ksw}, which we include here for easy future reference.

\begin{theorem}\label{theo:sropenarrow}\cite[Proposition 4.12]{ksw}
Let $X$ be a Banach space with the Daugavet property, $Y$ be a Banach space and $T:X\longrightarrow Y$ be a narrow operator. Then, for every finite-dimensional subspace $E$ of $X$, every $y\in B_X$, every weakly open subset $W$ of $B_X$ containing $y$ and every $\varepsilon>0$ we can find $x\in W$ such that:
\begin{enumerate}
\item $\Vert e+\lambda x\Vert>(1-\varepsilon)(\Vert e\Vert+\vert\lambda\vert)$
holds for every $e\in E$ and every $\lambda\in\lambda$; and 
\item $\Vert T(x-y)\Vert<\varepsilon$.
\end{enumerate}
\end{theorem}

The previous theorem allows us to get the following result.

\begin{theorem}\label{theo:sepasrLornarrow}
Let $X$ be a separable Banach space, $Y$ be a Banach space and $T:X\longrightarrow Y$ be narrow operator. Then, given any $y\in B_X$ and any subset $\{g_n: n\in\mathbb N\}\subseteq S_{X^*}$ we can find $u\in S_{X^{**}}$ such that
\begin{enumerate}
\item\label{ortonarrowsep} $\Vert x+u\Vert=1+\Vert x\Vert$ holds for every $x\in X$.
\item\label{iguanarrowsep} $T^{**}(u)=T(y)$.
\item $u(g_n)=g_n(y)$ holds for every $n\in\mathbb N$.
\end{enumerate}
In particular, given any non-empty $w^*$-open subset $W$ of $B_{X^{**}}$ there exists $u\in W$ satisfying \eqref{ortonarrowsep} and \eqref{iguanarrowsep}.
\end{theorem}

\begin{proof}
Pick $\{O_n\}_{n\in\mathbb N}$ a basis of the $b_X$-topology for $B_X$. By Theorem \ref{theo:sropenarrow} we can find, for every $n\in\mathbb N$, an element $x_n\in \{z\in B_X: \vert g_n(z-y)\vert<\frac{1}{n}:  1\leq i\leq n\}\cap \bigcap\limits_{k=1}^n O_k$ such that $\Vert T(x_n-y)\Vert<\frac{1}{n}$. Up passing to a subsequence, by Lemma \ref{lema:gklemabxabier}, we can assume that every cluster point $u$ of $\{x_n\}$ in the $w^*$-topology of $B_{X^{**}}$ satisfies that
$$\Vert x+u\Vert=1+\Vert x\Vert$$
holds for every $x\in X$. Take one such cluster point $u$. From the condition on the sequence it is clear that $u(g_n)=g_n(y)$ holds for every $n\in\mathbb N$. Let us prove that $T^{**}(u)=T(y)$. To this end, take a subnet $\{x_s\}$ of $\{x_n\}$ such that $x_s\rightarrow u$ in the $w^*$-topology of $B_{X^{**}}$. Then, the $w^*-w^*$ continuity of $T^{**}$ implies that
$$T(x_s)=T^{**}(x_s)\rightarrow T^{**}(u),$$
where the last convergence is in the $w^*$-topology of $Y^{**}$. Hence $T^{**}(x_s)-T(y)\rightarrow T(u)-T(y)$ in the $w^*$-topology. The $w^*$-lower semicontinuity of the norm of $Y^{**}$ implies that
$$\Vert T^{**}(u)-T(y)\Vert\leq \liminf \Vert T(x_s)-T(y)\Vert=0,$$
from where $T^{**}(u)=T(y)$, and we are done.
\end{proof}

We do not know whether the previous Theorem holds for $\dens(X)=\omega_1$. Note that, in order to use an inductive argument similar to the one of \cite[Theorem 3.3]{loru}, we would need to guarantee that Theorem \ref{theo:sepasrLornarrow} holds if we take $y\in B_{X^{**}}$. Let us exhibit, however, a class of narrow operators where Theorem \ref{theo:sepasrLornarrow} extends to the non-separable case.

\begin{theorem}\label{theo:nosepLordualsep}
Let $X$ be a Banach space with the Daugavet property and $Y$ be any Banach space. Let $T:X\longrightarrow Y$ be a bounded operator such that $T^*(Y^*)$ is separable. Then, given $\{g_n:n\in\mathbb N\}\subseteq S_{X^*}$ and any $u\in B_{X^{**}}$ we can find $v\in S_{X^{**}}$ such that
\begin{enumerate}
\item\label{ortonarrownonsep} $\Vert x+u\Vert=1+\Vert x\Vert$ holds for every $x\in X$.
\item\label{iguanarrownonsep} $T^{**}(v)=T^{**}(u)$.
\item $v(g_n)=u(g_n)$ holds for every $n\in\mathbb N$.
\end{enumerate}
In particular, given any non-empty $w^*$-open subset $W$ of $B_{X^{**}}$ there exists $u\in W$ satisfying \eqref{ortonarrownonsep} and \eqref{iguanarrownonsep}.
\end{theorem}

\begin{remark}\label{remark:adjuntosepnarrow}
It is known that, if $X$ is a Banach space with the Daugavet property and $Y$ be any Banach space, then any bounded operator $T:X\longrightarrow Y$ such that $T^*(Y^*)$ is separable satisfies that $T$ is narrow.

Indeed, if $A$ is any closed, bounded and convex subset of $B_X$, then $T(A)$ is \textit{slicely countably determited} by a sequence of slices $S(T(A),y_n^*,\delta_n)$, where $\{y_n^*\}\subseteq S_{Y^*}$ satisfies that $T^*(y_n^*)$ is dense in $T^*(S_{Y^*})$ and $\delta_n$ is a null sequence of positive scalars (see \cite{akmms} for formal definition). This implies, in the languaje of \cite{akmms}, that $T$ is \textit{hereditary slicely countably determined} and so, by \cite[Theorem 5.11]{akmms}, $T$ is narrow. We thank Miguel Mart\'in for pointing out this remark.
\end{remark}

\begin{proof}[Proof of Theorem \ref{theo:nosepLordualsep}]
Take $\{y_n^*\}\subseteq B_{Y^*}$ be such that $\{T^*(y_n^*)\}$ is dense in $T^*(B_{Y^*})$. Now we can apply \cite[Theorem 3.3]{loru} to find $v\in S_{X^{**}}$ satisfying that $u=v$ on $\{g_n:n\in\mathbb N\}\cup\{T^*(y_n^*): n\in\mathbb N\}$ and such that
$$\Vert x+v\Vert=1+\Vert x\Vert$$
holds for every $x\in X$. In only remains to prove that $T^{**}(v)=T^{**}(v)$. To this end notice that, since $\{T^*(y_n^*)\}$ is dense in $T^*(B_{Y^*})$, a density argument implies that $v=u$ on $T^*(B_{Y^*})$. Hence
$$T^{**}(v)(y^*)=v(T^*(y^*))=u(T^*(y^*))=T^{**}(u)(y^*)$$
holds for every $y^*\in B_{Y^*}$. From here it is inmediate to get that $T^{**}(v)=T^{**}(u)$, and the proof is complete.
\end{proof}

Note that for $\dens(X)>\omega_1$ the result is no longer true.

\begin{example}\label{exam:counnarrow}
Let $X=\ell_2(\mathcal P(\mathbb R))\iten C([0,1])$, $Y$ be any Banach space and $T:X\longrightarrow Y$ be the zero operator. Then $T$ is narrow since $X$ has the Daugavet property. However there is no $u\in S_{X^{**}}$ such that $\Vert x+u\Vert=1+\Vert x\Vert$ holds for every $x\in X$, so $T$ can not satisfy the thesis of Theorem \ref{theo:sepasrLornarrow}.

Note that, under continuum hypothesis, $\dens(\ell_2(\mathcal P(\mathbb R)))=card(\mathcal P(\mathbb R))=\omega_2$, and so the result is sharp.
\end{example}

\textbf{Acknowledgements:} The author thanks V. Kadets and M. Mart\'in for fruitful conversations on the topic of the paper.

\end{document}